\documentclass[12pt]{amsart}
\usepackage{geometry}                
\usepackage{graphicx}
\usepackage{amssymb}

\usepackage{amsthm}
\usepackage{amscd}
\usepackage[all]{xy}
\usepackage{pinlabel}
\usepackage[abs]{overpic}
\usepackage{comment}
\usepackage{subcaption}
\usepackage[dvipsnames]{xcolor} 
\usepackage{ulem}

\newcommand{\Z}{\mathbb{Z}}

\newcommand{\mL}{\mathcal{L}}
\newcommand{\tmL}{\widetilde{\mathcal{L}}}
\newcommand{\mC}{\mathcal{C}}
\newcommand{\mB}{\mathcal{B}}

\newcommand{\mP}{\mathcal{P}}
\newcommand{\tmP}{\widetilde{\mathcal{P}}}

\DeclareMathOperator{\Arf}{Arf}
\DeclareMathOperator{\lk}{lk}

\DeclareMathOperator{\sgn}{sgn}
\DeclareMathOperator{\Ker}{Ker}

\theoremstyle{theorem}
\newtheorem{thm}{Theorem}[section]
\newtheorem{prop}[thm]{Proposition}
\newtheorem{cor}[thm]{Corollary}
\newtheorem{lem}[thm]{Lemma}

\theoremstyle{definition}
\newtheorem{dfn}[thm]{Definition}

\theoremstyle{remark}
\newtheorem{remark}[thm]{Remark}

\begin{document}

\begin{abstract}
Donald and Owens introduced two link concordance groups with a marked component 
and showed that they contain the knot concordance group as a direct summand with infinitely generated complements. While not explicitly posed by Donald and Owens, the problem of determining the structure of these complements arises naturally from their work. 
In this paper, we completely resolve this problem by proving that both complements are isomorphic to 
\(\mathbb{Z}^{\infty} \oplus (\mathbb{Z}/2\mathbb{Z})^{\infty}\).
Moreover, we introduce a notion of prime element and establish a unique prime decomposition theorem. 
This yields a canonical normal form, providing a complete description of the group structure.
\end{abstract}

\title{Structure and unique factorization in concordance groups of links}

\author{Kouki Sato}
\address{Department of Mathematics, Meijo University, 1-501 Shiogamaguchi, Tempaku-ku, Nagoya 468-8502, Japan}
\email{satokou@meijo-u.ac.jp}

\author{Akira Yasuhara}
\address{Faculty of Commerce, Waseda University, 1-6-1 Nishi-Waseda,
  Shinjuku-ku, Tokyo 169-8050, Japan}
	 \email{yasuhara@waseda.jp}

\maketitle

\section{Introduction}\label{sec1}
A {\it partly oriented link} is a link in $S^3$ with a marked oriented component and the remaining components unoriented.
A {\it marked oriented link} is an oriented link in $S^3$ with a marked component.
If $L$ and $L'$ are  partly oriented links (resp.\ marked oriented link), then the connected sum $L \# L'$ with respect to the marked components is well-defined and commutative up to isotopy, and also regarded as a partly oriented link 
(resp.\ marked oriented link). In addition, 
the mirror $L^*$ and the orientation reversal  $-L$ of $L$
are also regarded as a partly oriented link (resp.\ marked oriented link).
On such objects, Donald-Owens \cite{DO} defined the following equivalence relation.

\begin{dfn}[\text{\cite[Definition 2]{DO}}]
\label{dfn:DO}
Let $L_0$ and $L_1$ be partly oriented links (resp.\ marked oriented links). 
We say $L_0$ and $L_1$ are 
{\it $\chi$-concordant}
if there is  
a smoothly,  properly embedded compact surface $F$ in the 4-ball $B^4$ 
without closed components, and with Euler characteristic $\chi(F)=1$ such that
\begin{itemize}
\setlength{\itemsep}{2mm} 
\item
$F$ consists of a single disk $D$ and unoriented surfaces (resp. oriented surfaces)
that are not disks,
\item $\partial F=-L^*_0 \# L_1$, and $\partial D$ is the marked component of $-L^*_0 \# L_1$.
\end{itemize}
\end{dfn}

In the definition above, since $\chi(F\setminus D)=0$,  we note that 
$F\setminus D$ consists of  annuli and/or M\"{o}bius bands in  the partly oriented case, 
and  of oriented annuli in the marked oriented case.

We denote by $\mL$ (resp.\ $\tmL$) the set of  $\chi$-concordance classes of partly oriented links (resp.\ marked oriented links), and call them the {\it $\chi$-concordance groups}. 
As proved in \cite{DO}, 
$\mL$ and $\tmL$ form abelian groups under connected sum respectively. 
Here, the trivial element $0$ in both $\mL$ and $\tmL$ is represented by the trivial knot.
Every trivial link represents $0$ in $\mL$. In contrast, a trivial link represents
$0$ in $\tmL$ if and only if the number of its components is odd, since any two marked
oriented links in the same class of $\tmL$ have numbers of components congruent modulo $2$.
Since a link $L$ is $\chi$-concordant to $K\#((-K)^*\#L)$,  
the equivalent class $[L]$ is equal to 
$[K]+[(-K)^*\#L]$ in each of the $\chi$-concordance groups $\mL$ and $\tmL$.
It is shown in \cite{DO} that this observation induces direct sum decompositions of $\mL$ and $\tmL$: 
\[
\mL = \mC \oplus \mL_0  \quad \text{and} \quad
\tmL = \mC \oplus \tmL_0, 
\]
where $\mC$ denotes the knot concordance group and $\mL_0$ and $\tmL_0$ consist of $\chi$-concordance classes of links whose marked component is a slice knot.
This implies that the intrinsic properties of $\mL$ (resp.\ $\tmL$) are concentrated in $\mL_0$ (resp.\ $\tmL_0$).

Since $\mL_0$ and $\tmL_0$ still contain classes of links with non-slice 
unmarked components, it appears very difficult to determine these groups, 
as determining $\mC$ is a widely open problem. 
Surprisingly, we obtain the following unexpected result, which determines 
their isomorphism classes.

\begin{thm}[Structure Theorem]
\label{thm:main}
Both $\mL_0$ and $\tmL_0$ are isomorphic to $\Z^\infty \oplus (\Z/2\Z)^\infty$. 
\end{thm}

In \cite{DO}, it is shown that both $\mL_0$ and $\tmL_0$ are infinitely generated 
and contain subgroups isomorphic to $\Z^\infty \oplus (\Z/2\Z)^\infty$.
However, it remained completely open to determine the isomorphism classes of these groups.
For example, it was unknown which orders of elements occur in these abelian groups.
Theorem~\ref{thm:main} settles these problems completely.

In the following, the notation $\mL_0^*$ denotes either $\mL_0$ or $\tmL_0$,
and {\it marked links} (or simply {\it links})
refers either partly oriented links or marked oriented links.
We assume that the marked component of any marked link is slice knot.

For proving Theorem~\ref{thm:main}, we introduce the {\it primeness} of elements in $\mL^*_0$. 
An element $x$ in $\mL^*_0$ is {\it non-prime} if $x$ can be realized by the connected sum $L_1 \# L_2$ of two marked links $L_1$ and $L_2$ such that
\begin{itemize}
\setlength{\itemsep}{2mm}
\item each $L_i$ has at least two components, including a marked slice component, and
\item the number of components of $L_1 \# L_2$ is minimal among all representatives of $x$.
\end{itemize}
An element $x \in \mL^*_0$ is {\it prime} if $x$ is neither zero nor non-prime. 
We stress that a non-trivial element $x\in \mL^*_0$ containing a 2-component link 
is prime. A remarkable consequence of defining 
primeness is the following.
 
\begin{thm}[Prime Decomposition Theorem]
\label{thm:prime}
Any non-zero element $x \in \mL^*_0$ can be written in the form $x = \sum_{i=1}^n k_i x_i$ where $x_i$ is prime,
$k_i x_i \neq 0$ and $x_s \neq \pm x_t$ for any $1 \leq i \leq n$ and $1 \leq s < t \leq n$.
Moreover, the form  is uniquely determined up to order of terms. 
\end{thm}

In the next section, we show that any prime element has order 2 or $\infty$ 
(Corollary~\ref{cor:prime}). These already imply that 
$\mL^*_0 \cong \Z^{\alpha} \oplus (\Z/2\Z)^{\beta}$ for some cardinalities $\alpha, \beta$.
(Note that since the set of the isotopy classes of marked links is countable, $\alpha$ and $\beta$ are at most countable.)
To prove $\alpha = \beta = \infty$, one needs to find an infinite family of mutually distinct prime elements in 
$\mL^*_0$. 
Since $\mL^*_0$ still contains links with knotted unmarked components, it is not hard to see 
 $\alpha = \beta = \infty$ by adding local knots to unmarked components. 
 Roughly speaking, the following approach can be made: 
 \begin{itemize}
\item (for $\alpha=\infty$)~ Let $K$ be a knot with infinite order in $\mC$. 
A link  obtained from the Hopf link by adding $K$ to the unmarked component
 is prime and infinite order in $\mL^*_0$. 
 \item(for $\beta=\infty$)~
 Since it is known in \cite{Yasuhara} that any knot with trivial signature and non-trivial Arf invariant does not 
 bound a M\"{o}bius band in the 4-ball, 
we have infinitely many knots with order 2 in $\mC$ 
that do not bound a M\"{o}bius band in the  4-ball. 
The split union of the marked trivial knot and such a knot 
is prime and has order 2 in $\mL^*_0$.  
 \end{itemize}
 
It follows that, by considering local knots, we can prove Theorem~\ref{thm:main} more simply than in the proof presented in this article. 
On the other hand, by avoiding local knots, we can observe how our objects are linked without the \lq noise' introduced by knots. 
This suggests that investigating links with trivial components is a reasonable approach. 
Furthermore, by focusing on {\it Brunnian links}, we gain deeper insight into the structure of such links.
Here, a {\it Brunnian link} $L$ is a non-trivial link which becomes trivial if any component is removed. 

Two $n$-component marked links $L_0 \subset S^3 \times \{0\}$ and $L_1 \subset S^3 \times \{1\}$ with marked components $K_0$ and 
$K_1$ respectively are {\it marked-concordant} if there
 is a disjoint union $A$ of annuli $A_1,...,A_n$ in $S^3\times[0,1]$ such that 
$\partial A_j \cap (S^3 \times \{i\}) \neq \emptyset$ 
$(i=0,1,~j=1,\dots,n)$, $A_1$ is oriented with $\partial A_1 = K_0 \cup (-K_1^*)$, 
and the following conditions hold:
\begin{itemize}
\setlength{\itemsep}{2mm}
\item In the partly oriented case,  $A\setminus A_1$ is unoriented and 
$\partial (A\setminus A_1) = (L_0\setminus K_0) \cup (L_1\setminus K_1)^*$.
\item In the marked oriented case, 
$\partial A = L_0 \cup (-L_1^*)$.
\end{itemize}

Brunnian links have the following property.

\begin{prop}
\label{prop:Brunnian}
If a Brunnian link $L$ is not equal to any trivial link in $\mL^*_0$, then $L$ represents a prime element in $\mL^*_0$.
Moreover, if another Brunnian link $L'$ also represents the same element, then $L$ and $L'$ are 
marked-concordant.
\end{prop}

As concrete examples, we consider Brunnian links $W_n$ and $B_n$ shown in Figure~\ref{Wn_and_Bn}, and prove the following.

\begin{prop}
\label{prop:example}
For any odd integer $n > 0$, the marked oriented link $W_n$ (resp.\ $B_n$) is not equal to any trivial link and has order $\infty$ (resp.\ order 2) in $\mL^*_0$. Moreover, 
$W_{2k-1}~(k \in \mathbb{N})$ (resp.\ $B_{2k-1}~(k \in \mathbb{N})$) are mutually distinct up to $\chi$-concordance. 
\end{prop}

\begin{remark}
We note that $W_1$ and $B_1$ are the Whitehead link and the Borromean rings respectively.
In particular, Proposition~\ref{prop:example} answers to Donald-Owens' question \cite[Section 4]{DO} asking what the orders of $W_1$ and $B_1$ are.
\end{remark}

\begin{figure}[htbp]
\begin{center}
\includegraphics[width=.85\linewidth]{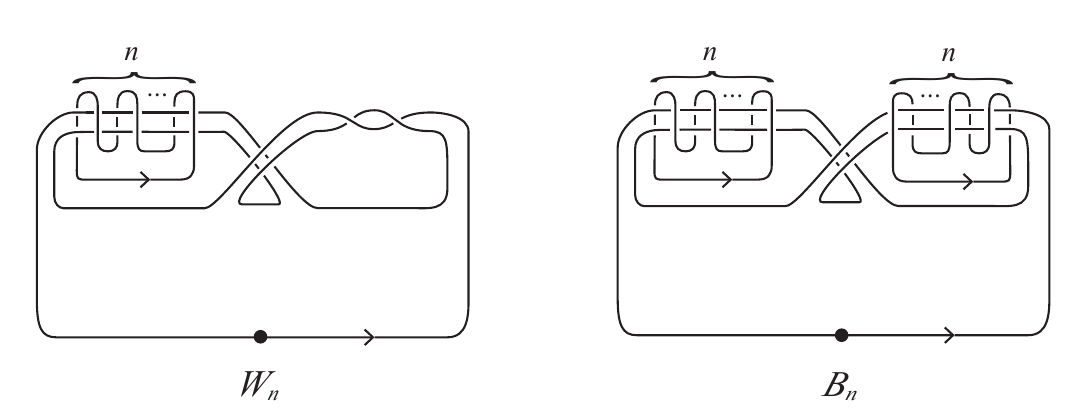}
\caption{Marked oriented links $W_n$ and $B_n$}\label{Wn_and_Bn}
\end{center}
\end{figure}

By ignoring the orientations on the unmarked component of a marked oriented link, 
we have the natural epimorphism $\rho: \tmL_0\longrightarrow \mL_0$. 
We also regard $\rho$ as a map from the set of marked oriented links to 
the set of partly oriented links.
The structure of $\ker \rho$ is described as follows.

\begin{thm}
\label{thm:ker}
The kernel of the projection $\rho:\widetilde{\mathcal{L}}_0 \to \mathcal{L}_0$ is isomorphic to 
$\mathbb{Z}^\infty \oplus (\mathbb{Z}/2\mathbb{Z})^\infty$. 
Furthermore, let $\mathcal{B}$ denote the set of nontrivial elements in $\widetilde{\mathcal{L}}_0$ represented by Brunnian links. 
Then $\widetilde{\mathcal{L}}_0$ contains the subgroup generated by $(\ker \rho) \cap \mathcal{B}$ as a direct summand, 
which is isomorphic to $\mathbb{Z}^\infty \oplus (\mathbb{Z}/2\mathbb{Z})^\infty$.
\end{thm}


\section{Prime decomposition in $\mL^*_0$}

In this section, we prove Theorem~\ref{thm:prime} and Proposition~\ref{prop:Brunnian}.
We first introduce several notions and prove lemmas.
For marked links $L_0$ and $L_1$, 
a properly embedded surface $F$ in $B^4$ is called a {\it $\chi$-concordance from $L_0$ to $L_1$} if 
$F$ satisfies the conditions in Definition~\ref{dfn:DO}.
Throughout this section, for given marked link $L$  with marked component $K$, 
we denote by $\check{L}$ the sublink $L\setminus K$ of $L$.
Then, $-L^*_0 \# L_1$ can be regarded as the union of  the marked component 
and $(-\check{L_0^*})\cup\check{L_1}$.
A marked link $L$ is a {\it minimizer} of $[L] \in \mL^*_0$
if the number of the components of $L$ is minimal among all representatives of $[L]$.
Let $c(L)$ denote the number of the components of $L$.

\begin{lem}
\label{lem:minimal1}
Let $L$ and $L'$ be $\chi$-concordant links. 
If $L$ is a minimizer of $[L]\in\mL_0^*$, then there does not exist a $\chi$-concordance $F$ from $L'$ to $L$ 
such that $F$  contains a connected component bounded by a sublink of $\check{L}$.
\end{lem}

\begin{proof}
If  a $\chi$-concordance $F$ from $L'$ to $L$ contains 
a connected component $F'$ with $\partial F' \subset \check{L}$, then 
$F \setminus F'$ is a $\chi$-concordance from $L'$ to $L \setminus \partial F'$. 
This contradicts the assumption that $L$ is a minimizer, since $L$ and $L'$ are $\chi$-concordant.
\end{proof}

\begin{lem}
\label{lem:minimal2}
If $\chi$-concordant links $L$ and $L'$ are minimizers of $[L] =[L'] \in \mL^*_0$, then 
a $\chi$-concordance $F$  from $L'$ to $L$ induces a marked-concordance from $L'$ to $L$.
\end{lem}

\begin{proof}
It immediately follows from Lemma~\ref{lem:minimal1} that 
$F$ consists of one disk and several annuli, each of which connects 
a component of $(-L')^*$ to a component of $L$. 
Since $F$ is bounded by $(-L')^*\#L$, by the definition of the connected sum of marked links, 
$F$ induces a marked-concordance from $L'$ to $L$ in a canonical way.
\end{proof}

\begin{lem}
\label{lem:sum1}
Let  $L_1$ and $L_2$ be marked links representing prime elements in $\mL^*_0$ with  
$[L_1 \# L_2]\neq 0$.
If both $L_1$ and $L_2$ are minimizers, then 
$L_1 \# L_2$ is a minimizer of $[L_1 \# L_2]$.
\end{lem}

\begin{proof}
Let $L$ be a minimizer of $[L_1 \# L_2] \in \mL^*_0$.
Suppose that $L_1\# L_2$ is not a minimizer. 
Then $c(\check{L}) < c(\check{L_1}) + c(\check{L_2})$. 
Let $F$ be a $\chi$-concordance from $L$ to $L_1 \# L_2$ and $D$ a disk-component of $F$. 
We note that $F$ yields $\chi$-concordances from $L \# (-L_1^*)$ to $L_2$, and 
from $L \# (-L_2^*)$ to $L_1$.
Since $L_1$, $L_2$, and $L$ are minimizers, by Lemma~\ref{lem:minimal1}, 
$F\setminus D$ consists of annuli, none of which is bounded by any sublink of 
$\check{L}$, $\check{L_1}$, or $\check{L_2}$.

On the other hand, since $[L]=[L_1 \# L_2] \neq 0$ in $\mathcal{L}_0^*$, we have $c(L) \geq 2$. 
Hence $F\setminus D$ contains a union of annuli such that the boundary contains $\check{L}$, 
and  each component connects a component of $\check{L}$ to a component of  
$\check{L_1}\cup\check{L_2}$.  Furthermore, we may assume that there is at least one annulus in 
$F\setminus D$ connecting a component of $\check{L}$ to a component of  $\check{L_1}$.
Let $A_1(\subset F\setminus D)$ be the union of these annuli. 

Since $c(\check{L}) < c(\check{L_1}) + c(\check{L_2})$, 
$F$ contains at least one annulus  connecting a component of $\check{L_1}$ to
a component of $\check{L_2}$. Let $A$ be the union of these annuli.

It follows that we obtain  a $\chi$-concordance $D \cup A \cup A_1$ from $L_1$ to the connected sum
$L' \# L_2'$ of sublinks $L' (\subset L)$ and $L_2' (\subset L_2)$. 
Since $c(L') \geq 2$, $c(L_2') \geq 2$, and $c(L' \# L_2') = c(L_1)$, 
this contradicts either that $[L_1]$ is prime or that $L_1$ is a minimizer.
\end{proof}

\begin{prop}
\label{lem:sum2}
Let $\{ x_i \}_{i=1}^n$ be a finite sequence of prime elements in $\mL_0^*$ 
(possibly with the same element appearing multiple times in the sequence). 
If  
$x_i + x_j \neq 0$ for any $1 \leq i < j \leq n$.
then, $\sum_{i=1}^n x_i\neq 0\in\mL_0^*$.
\end{prop}

\begin{proof}
Let $L_i$ be a minimizer of $[L_i] =x_i$ $(i=1,...,n)$. 
Suppose that  $\sum_{i=1}^n x_i = 0$. 
Then there exists a $\chi$-concordance $F$ with the disk-component $D$ 
from the unknot to 
$\#_{i=1}^n L_i$. Here, we see 
\[
c(\#_{i=1}^n L_i) = \sum_{i=1}^n c(\check{L_i}) + 1 \geq n+1 \geq 2,
\]  
and hence $F\setminus D$ contains at least one connected component $C$ 
such that $\partial C$ is contained in $\check{L_i}\cup \check{L_j}$ for some $1\leq i<j\leq n$. 
(Note that in the partly oriented case, $C$ might be a M\"{o}bius band 
with its boundary contained in a single link $\check{L_i}$.)
Since $F$ yields a $\chi$-concordance from $\#_{k \neq i,j}(-L^*_k)$ to $L_i \# L_j$, 
\[
[(L_i \# L_j )\setminus \partial C] = [\#_{k \neq i,j}(L_k)] = [L_i \# L_j].
\]
This is a contradiction, since, by Lemma~\ref{lem:sum1}, $L_i \# L_j$ is a minimizer.
\end{proof}

\begin{cor}
\label{cor:prime}
Any prime element of $\mL^*_0$ has order 2 or $\infty$.
\end{cor}

\begin{proof}
Let $x\in \mL^*_0$ be a prime element with $2x \neq 0$.
Then, for any $n \geq 1$,
we can apply Proposition~\ref{lem:sum2} to a sequence $\{x_i\}_{i=1}^n = \{x, x, \ldots, x\}$,
which shows $nx \neq 0$.
\end{proof}

Set 
\[
\mP^*_2 := \{ x \in \mL^*_0 \mid \text{ $x$ is prime and has order 2} \}
\]
and
\[
\widetilde{\mP^*_{\infty}} := \{ x \in \mL^*_0 \mid \text{ $x$ is prime and has order  $\infty$} \}.
\]
Consider an equivalence relation on $\widetilde{\mP^*_\infty}$ defined as follows: 
$x \sim y$ if $x = \pm y$ in $\mathcal{L}_0^*$. 
Then take a subset $\mathcal{P}^*_\infty \subset \widetilde{\mP^*_\infty}$ by choosing 
exactly one representative from each equivalence class of $\widetilde{\mP^*_\infty} / \sim$.
Let $\mL^*_{\infty}$ and $\mL^*_2$ denote the subgroups of $\mL^*_0$ generated by $\mP^*_{\infty}$ 
and $\mP^*_{2}$, respectively.
The following theorem essentially determines the structure of $\mL^*_0$.

\begin{thm}
\label{thm:structure}
The following assertions hold:
\begin{itemize}
\setlength{\itemsep}{2mm}
\item[(1)]
$\mL^*_0 = \mL^*_\infty \oplus \mL^*_2$,
\item[(2)]
$\mL^*_\infty$ is a free abelian group with a free basis $\mP^*_\infty$, and
\item[(3)]
$\mL^*_2$ is a $\Z/2\Z$-vector space with a basis $\mP^*_2$.
\end{itemize}
\end{thm}

\begin{proof}
We prove (1).
First we prove that $\mL^*_0 = \mL^*_\infty + \mL^*_2$. 
By Corollary~\ref{cor:prime}, it is enough to show that 
any non-zero element $x \in \mL^*_0$ is decomposed into a sum of prime elements.

Let $c(x)$ denote the number of components of a minimizer of $x$.
We argue by induction on $c(x)$. 

If $c(x) =2$, then $x$ is prime, and hence the assertion holds.
Suppose  that $c(x) =n (\geq 3)$. 
If $x$ is prime, the assertion holds. Hence we may assume that  $x$ is non-prime. 
Then there exists marked links $L_1, L_2$ such that 
$L_1 \# L_2$ is a minimizer of $[L_1] + [L_2] =x$ and $c(L_1), c(L_2) \geq 2$.  
It follows that $c(L_1), c(L_2) < n$, and by the induction hypothesis, each of $[L_1]$ and $[L_2]$ 
can be decomposed into a sum of prime elements; therefore, so can $x$. 

Next, we prove $\mL^*_\infty \cap \mL^*_2 = \{ 0\}$.
Assume $\mL^*_\infty \cap \mL^*_2 \neq \{ 0\}$. Then
there exists a non-zero element $x\in \mL^*_\infty \cap \mL^*_2 $.
Since $x\in \mL^*_\infty$, 
we can represent $x$ as a linear combination of mutually distinct elements 
$x_1, \dots, x_m$ in $\mP^*_\infty$ as
\[
x = \sum_{i=1}^m k_i x_i,
\]
where $k_i \neq 0$ for $1 \leq i \leq m$.

On the other hand,  $2x=\sum_{i=1}^m 2k_i x_i$ can also be regarded as the summation
of a sequence 
\[\{y_{11},...,y_{1,2|k_1|},...,y_{m,1},...,y_{m,2|k_m|}\},\] 
where $y_{is}=(\sgn k_i) x_i~(i=1,...m,~s=1,...,2|k_i|)$.  
By the definition of $\mP^*_{\infty}$, 
\[y_{is}+y_{jt}=(\sgn k_i) x_i + (\sgn k_j) x_j \neq 0\] 
for any $1 \leq i \leq j \leq m,~1\leq s\leq 2|k_i|,~1\leq t\leq 2|k_j|$.
Then Proposition~\ref{lem:sum2} implies that 
\[2x=\sum_{i=1}^m 2k_i x_i=\sum_{i=1}^m\sum_{s=1}^{k_i}y_{is} \neq 0,\] 
which contradicts the fact that $x$ belongs to $\mL^*_2$.

Similarly, the assertions (2) and (3) are also proved by Proposition~\ref{lem:sum2}.
\end{proof}

Now, we show Theorem~\ref{thm:prime} and Proposition~\ref{prop:Brunnian}.

\begin{proof}[Proof of Theorem~\ref{thm:prime}]
The existence of the form $x = \sum_{i=1}^n k_i x_i~(x_i\in\mP^*_{\infty}\cup\mP^*_2)$ 
immediately follows from Theorem~\ref{thm:structure}~(1).

To prove the second-half assertion, 
suppose that $x$ is also represented as 
$x  = \sum_{i=1}^m l_i y_i$, 
where $y_i$ is prime, $l_i y_i \neq 0$ and 
$y_s \neq \pm y_t$ for any $1 \leq i \leq m$ and $1 \leq s < t \leq m$.
Then either $y_i$ or $-y_i$ belongs to $\mP^*_\infty \cup \mP^*_2$ for each $i$,
and hence Theorem~\ref{thm:structure}~(2),(3) implies 
$m = n$ and the existence of a permutation $\sigma$
such that $l_i y_i = k_{\sigma(i)} x_{\sigma(i)}$.
\end{proof}

\begin{proof}[Proof of Proposition~\ref{prop:Brunnian}]
Let $L$ be a marked Brunnian link which is not equal to any trivial link in $\mL^*_0$.
We first prove that $L$ is a minimizer of $[L] \in \mL^*_0$.
Assume that $L_0$ is a minimizer of $[L_0] =[L] \in \mL^*_0$ and $F$ is a $\chi$-concordance 
from $L_0$ to $L$.
If $L$ is not a minimizer, then $c(L_0)<c(L)$ and hence  
$F$ has a connected component $C$ with 
$\partial C \subset \check{L}$. Now, a surface $F \setminus C$ gives
\[
[L \setminus \partial C] = [L_0] = [L],
\]
where $L \setminus \partial C$ is a trivial link. 
This contradicts to the assumption, and hence $L$ must be a minimizer of $[L] \in \mL^*_0$.

Next, let us prove the primeness of $[L] \in \mL^*_0 $. Assume that $[L]$ is non-prime, and then there exist marked links $L_1, L_2$ such that each $L_i$ has at least 2 components and $L_1 \# L_2$ is a minimizer 
of  $[L_1 \# L_2] = [L] \in \mL^*_0$. 
Here, Lemma~\ref{lem:minimal2} 
gives a marked-concordance from $L_1 \# L_2$ to $L$.
Since $c(\check{L_1}), c(\check{L_2}) \geq 1$, marked-concordance gives rise to
 a marked-concordance from each $L_i$ to a sublink of $L$,
which is a trivial link. This implies that $[L] =[L_1 \# L_2]$ is equal to some trivial link in $\mL^*_0$, a contradiction.
Therefore, $[L]$ is prime.

Now, the existence of marked-concordance between any two Brunnian links in $[L]$ directly follows from Lemma~\ref{lem:minimal2}.
\end{proof}

\section{Brunnian-Link Basis Elements}
In this section, we prove Theorem~\ref{thm:main}, Proposition~\ref{prop:example} and Theorem~\ref{thm:ker}.
We first prove Proposition~\ref{prop:example}.

Let $W = O \cup K_1$ be the Whitehead link with marked component $O$ and $P_n \subset S^1 \times D^2$ a pattern with winding number $n$. Then we set $P_n(W) := O \cup P_n(K_1)$.
For any $n \in \mathbb{N}$, we denote  the $(n,1)$-cabling by $C_{n,1}$ and the $n$-fold $(2,1)$-cabling by $C^n$.
Then we see $C_{n,1}(W) = W_n$, and the assertion for $W_n$ in Proposition~\ref{prop:example} is a consequence of the following proposition.

\begin{prop}
\label{prop:W}
The following assertions hold:
\begin{itemize}
\setlength{\itemsep}{2mm}
\item[(1)] If $n \neq 0$, then $[P_n(W)]$ is prime and has infinite order  in $\tmL_0$. Moreover, if $n \neq m$,
then $[P_n(W)] \neq [P_m(W)]$ in $\tmL_0$.
\item[(2)] For any odd $n > 0$, $[\rho(C_{n,1}(W))]$ is prime and has infinite order in $\mL_0$. 
Moreover, $[\rho(C_{2k-1,1}(W))]~(k\in\mathbb{N})$ are mutually distinct in $\mL_0$.
\item[(3)] For any $n \geq 1$, $[\rho(C^n(W))]=0$  in $\mL_0$. 
\end{itemize}
\end{prop}

To prove Proposition~\ref{prop:W}, we need the following three lemmas. Here, for given knot $K$,
denote by $\sigma(K)$ the knot signature of $K$ and by $\Arf (K)$ the Arf invariant of $K$. 
\begin{lem}
\label{lem:2comp1}
Let $L$ be a 2-component marked oriented link and $K$ a knot with $\sigma(K) + 4 \Arf (K) \equiv 4  \mod 8$.
If there is a planer surface $F$ in $S^3 \times [0,1]$ from $L$ to $K$, then 
the partly oriented link $\rho(L)$ is non-zero in $\mL_0$.
\end{lem}

\begin{proof}
Suppose that $L = K_1 \cup K_2$ and $K_1$ is the marked component.
If $[\rho(L)] = 0$ in $\mL_0$, then $\rho(L)$ bounds a disjoint union $F'$ of a disk and 
a M\"{o}bius band in $B^4$.  
By gluing $(B^4, F')$ and $(S^3\times [0,1],F)$ along $\rho(L)$, 
we obtain a M\"{o}bius band $F\cup F'$  in $B^4$ with boundary $K$. 
However, it is proved in \cite{Yasuhara} that any knot $K$ with $\sigma(K) + 4 \Arf (K) \equiv 4 ( \mod 8)$
cannot bound a M\"{o}bius band in $B^4$, a contradiction. Thus we have $[L] \neq 0$ in $\mL_0$.
\end{proof}

\begin{dfn}\label{def:linkage}
Let $L= O \cup K_1\cup\cdots\cup K_n$ be an $(n+1)$-component marked oriented link such 
that the marked component $O$ is an unknot,  and $\lk(O, K_i)$ is even for any $i=1,...,n$. 
Then the double branched cover $\Sigma_2(O)$ of $S^3$ branched over $O$ 
is homeomorphic to $S^3$. 
Hence, the linking number of any pair of lifts $K_i$ and $K_j$ in $\Sigma_2(O)$  
takes values in $\Z$.
Since all $\lk(K, K_i)~(i=1,...,n)$ are even, $O$ bounds an unoriented surface $F$ in $S^3$ 
disjoint from $K_1\cup\cdots\cup K_n$. 
Then $\Sigma_2(O)$ decomposes as the union of $M_1$ and $M_2$, each of which is 
a copy of $S^3$ cut along $F$. 
For each $i,~k\in\{1,2\}$, let $K_{ik}$ denote the lift of $K_i$ contained in $M_k$.
In the following we denote $\lk(K_{ik},K_{jl})$ by $\lambda^F_L(ik,jl)$, or simply 
$\lambda_L(ik,jl)$. 

We note that $\lambda_L$ is symmetric, i.e., $\lambda_L(ik,jl)=\lambda_L(jl,ik)$.
We remark that if for some $i\in\{1,...,n\}$ the sublink $O\cup K_i$ is a split link, then 
$\lambda(i1,i2)=0$. 
\end{dfn}

\begin{remark}\label{cal:linkage}
In the case that $\lk(K_i,K_j)=0$ for $1 \le i\leq j \le m$  ($\lk(K_i,K_j)$ is defined to be $0$ if $i=j$), 
by \cite[Theorem~2.3]{PY}, we can calculate $\lambda^F_L(ik,jl)$ by using 
a {\it Goeritz matrix} ~\cite{G,GL}  of $F$ as follows:
Let $G_\alpha$ be the Goeritz matrix with respect to a basis 
$\alpha = (a_1, \dots, a_n)$ of $H_1(F)$, i.e., the $(i,j)$-entry of $G_\alpha$ is given by
$\mathrm{lk}(a_i, \tau a_j)$, 
where $\tau a_j$ is the 1-cycle in $S^3 - F$ obtained by pushing off $2 a_j$ in both normal directions.
For each $K_i$, define the vector
\[
V_\alpha(K_i) = (\mathrm{lk}(K_i, a_1), \dots, \mathrm{lk}(K_i, a_n)).
\]
Then  
\[
\lambda^F_L(ik,jl) = (-1)^{\delta_{kl}}V_\alpha(K_i) \, G_\alpha^{-1} \, V_\alpha(K_j)^T,
\] 
where $\delta_{kl}$ is the Kronecker delta.
\end{remark}

\begin{lem}
\label{lem:2comp2}
Let $L$ and $L'$ be 2-component marked links such that 
their marked components are unknots and their linking numbers are even.
If $[L]=[L']$ in $\tmL_0$, then 
$\lambda_L({11},{12}) = \lambda_{L'}({11},{12})$.
\end{lem}
\begin{proof}
Since both $L$ and $L'$ are $2$-component links, they are minimizers of $[L] =[L'] \in \tmL_0$.
Hence
Lemma~\ref{lem:minimal2} shows that there exists a marked-concordance 
$A \cup A_1$ from $L$ to $L'$
such that $A_1$ connects the unmarked components of $L$ and $L'$. 
Moreover, there are two lifts ${A}_{11}$, ${A}_{12}$ of $A_1$ in the double branched cover 
$\Sigma_2(A)$ of $S^3 \times [0,1]$ over $A_1$. 
Since $\Sigma_2(A)$ of $S^3 \times [0,1]$ is a two punctured rational homology 4-sphere, 
$|\lambda_L(11,12)-\lambda_{L'}(11,12)|=|A_{11}\cdot A_{12}|=0$, 
where $A_1 \cdot A_2$ is the intersection number of $A_1$ and $A_2$.
\end{proof}

\begin{lem} \label{lem:2comp3} 
Let $L$ and $L'$ be 2-component marked links as in Lemma~\ref{lem:2comp2}.
If $[\rho(L)] = [\rho(L')] \neq 0$ in $\mL_0$, then $\lambda_L({11},{12}) = \lambda_{L'}({11},{12})$.
\end{lem}

\begin{proof}
Let  $L = O \cup K_1$ and $L' = O \cup K'_1$, where $O$ is the marked component of each.
Since $[L]$ and $[L']$ are  non-zero elements in $\mL_0$, 
$L$ and $L'$ are minimizers.
Thus, we can adopt the same argument as the proof of Lemma~\ref{lem:2comp2}, except for that 
the annulus $A_1$ possibly connects $K_1$ to $-K'_1$. 
Even for that case, we reaches the same conclusion as the original case since 
\[\lambda_{O\cup K'_1}({11},{12}) =\lk(K_{11},K_{12}) = \lk(-{K}_{11}, -{K}_{12})= \lambda_{O\cup(- K'_1)}({11},{12}).\]
\end{proof}

\begin{proof}[Proof of Proposition~\ref{prop:W}]
We first prove the assertion (1).  
Since $P_n(W)$ is a 2-component link, $[P_n(W)] \in \tmL_0$ is prime.  
Thus, by Corollary~\ref{cor:prime}, we only need to show $[P_n(W)] \neq [-P_n(W)^*]$
and $[P_n(W)] \neq [P_m(W)]$ in $\tmL_0$.
Let $L=O\cup K_1$ denote $P_n(W)$.
If  $[P_n(W)] = [-P_n(W)^*]$, then Lemma~\ref{lem:2comp2} shows 
\[
\lambda_{L}(11,12) = \lambda_{-L^*}(11,12).\] 
Since $(S^3,-L^*)\cong (-S^3, -L)$, this homeomorphism induces 
 \[\lambda_{-L^*}(11,12)=-\lk(-K_{11},-K_{12})~(=-\lambda_{L}(11,12)).\]
Hence we have  $\lambda_{L}(11,12) = 0$.

On the other hand, 
for a surface $F$ with $\partial F = O$, and for a basis 
$\alpha = (a_1, a_2)$ illustrated in Figure~\ref{lambda}, we have
\[
G_{\alpha} =
\begin{pmatrix}
0 & -1 \\
-1 & 2
\end{pmatrix}.
\]
Hence, by Remark~\ref{cal:linkage}, 
\[
\lambda_{W}(11,12)
= (1 \;\; 0)
G_{\alpha}^{-1}
\begin{pmatrix}
1 \\ 0
\end{pmatrix}
= (1 \;\; 0)
\begin{pmatrix}
-2 & -1 \\
-1 & 0
\end{pmatrix}
\begin{pmatrix}
1 \\ 0
\end{pmatrix}
= -2.
\]
Furthermore,   
\[
\lambda_{P_n(W)}(11,12)
= (n \;\; 0)
\begin{pmatrix}
-2 & -1 \\
-1 & 0
\end{pmatrix}
\begin{pmatrix}
n \\ 0
\end{pmatrix}
= -2 n^2.
\]
This shows $[P_n(W)] \neq [-P_n(W)^*]$ and also implies that if $n \neq m$ then $[P_n(W)] \neq [P_m(W)]$.

We next prove the assertion (2).
Figure~\ref{cobordism} shows that there is a planar surface in $S^3 \times [0,1]$ from $C_{n,1}(W)$ to a twist knot $K_n$,
where $\sigma(K_n) = 0$ and $\Arf (K_n) \equiv 1 \mod 2$ for any odd $n > 0$. 
Moreover, we have $w(C_{n,1}) = n$. Now, 
it follows from Lemma~\ref{lem:2comp1}, Lemma~\ref{lem:2comp3} and arguments in the previous paragraph
that the assertion (2) holds. 

Finally, one can easily check the assertion (3) by constructing a disjoint union of a disk and a M\"{o}bius band in $B^4$ with boundary $\rho(C^n(W))$. 
\end{proof}

\begin{figure}[htbp]
\begin{center}
\includegraphics[width=.85\linewidth]{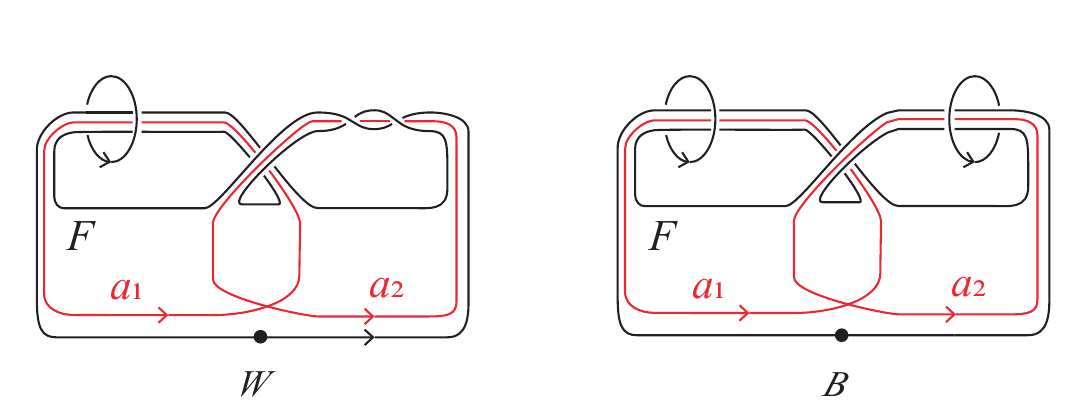}
\caption{Surfaces $F$ bounded by marked components and a basis $a_1, a_2$ of $H_1(F)$}\label{lambda}
\includegraphics[width=.85\linewidth]{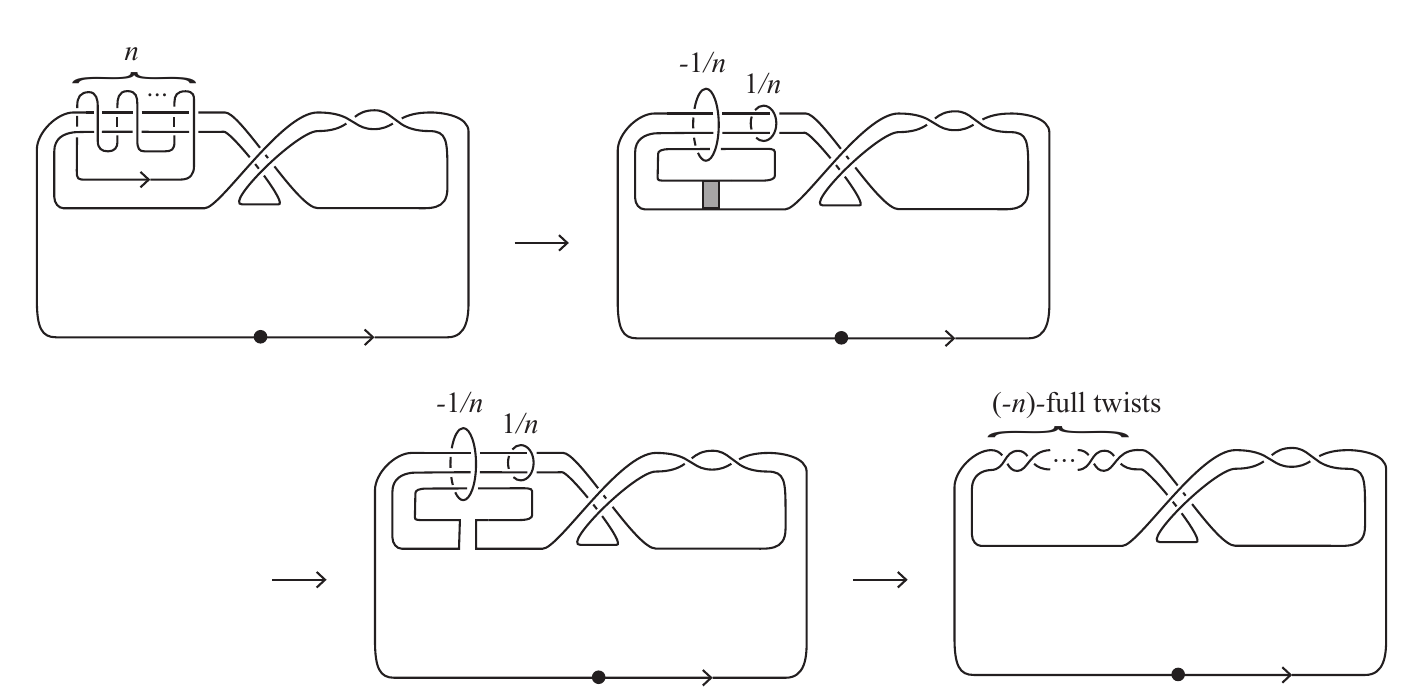}
\caption{A planar surface in $S^3 \times [0,1]$ from $C_{n,1}(W)$ to a twist knot $K_n$}\label{cobordism}
\end{center}
\end{figure}

Next, let us consider the assertion for $B_n$ in Proposition~\ref{prop:example}.
Let $B = O \cup K_1 \cup K_2$ be the Borromean rings with marked component $O$ and 
$P_n \subset S^1 \times D^2$ a pattern with winding number $n$. 
Set $P_n(B) := O \cup P_n(K_1) \cup P_n^*(K_2)$,
where $P_n^*$ is the mirror of $P_n$.
Then we see $C_{n,1}(B) = B_n$, and the assertion for $B_n$ in Proposition~\ref{prop:example} is a consequence of the following proposition.

\begin{prop}
\label{prop:B}
The following assertions hold:
\begin{itemize}
\setlength{\itemsep}{2mm}
\item[(1)] For any $P_n$, we have $2[P_n(B)] = 0$ in $\tmL_0$.

\item[(2)] For any odd $n > 0$, $[C_{n,1}(B)]$ and  $[\rho(C_{n,1}(B))]$ are prime in $\tmL_0$ and 
$\mL_0$, respectively. 
Moreover, $[\rho(C_{2k-1,1}(B))]~(k\in\mathbb{N})$ are mutually distinct in $\mL_0$.
\item[(3)] For any $n \geq 1$, $[\rho(C^n(B))]=0$  in $\mL_0$, and  $[C^n(B)]$ is prime in $\tmL_0$. 
Moreover, $[C^n(B)]~(n\in\mathbb{N})$ are mutually distinct in $\tmL_0$.
\end{itemize}
\end{prop}

To prove Proposition~\ref{prop:B}, we first show lemmas for 3-component links.

\begin{dfn}\label{def:linkage2}
Let $L= O \cup K_1\cup\cdots\cup K_n$ be an $(n+1)$-component marked oriented link as  
in Definition~\ref{def:linkage}, i.e., 
the marked component $O$ is an unknot and $\lk (O, K_i)$ is even for any $i =1,...,n$.
Furthermore, suppose that $\lk(K_i,K_j)=0$ for some $i \neq j$. 
Then by Remark~\ref{cal:linkage}, we have 
$\lambda(i1,j1)=\lambda(i2,j2)=-\lambda(i1,j2)=-\lambda(i2,j1)$.  
Hence we define $\bar{\lambda}(i,j)$ as
\[\bar{\lambda}(i,j):=|\lambda(i1,j1)| .\]
\end{dfn}

In the following lemmas, we consider Brunnian links. We stress that 
Brunnian links satisfy the 
condition in the definition above.

\begin{lem}
\label{lem:3comp1}
If $L$ is Brunnian and $[L] = 0$ in $\tmL_0$, then
\[
\bar{\lambda}_L(1,2)=0.\]
\end{lem}

\begin{proof}
Since $L$ is Brunnian,  for each $i$,  
$O\cup K_i$ is a trivial link, and hence  $\lambda(i1,i2)=0$.

If $[L] = 0$ in $\tmL_0$, then $O$ and $K_1 \cup K_2$ bounds a disk $D$ and annulus $A$ in 
$B^4$ respectively, which are mutually disjoint. 
Moreover, 
there are two lifts ${A}_1$, ${A}_2$ of $A$ in the double branched cover 
$\Sigma_2(D)$ of $B^4$ branched over $D$.
Note that $\Sigma_2(D)$ is a rational homology 4-ball, and hence
$\lk (\partial {A}_1, \partial {A}_2) = {A}_1 \cdot {A}_2 = 0$.

Since $(\partial A_1, \partial A_2) = (K_{11} \cup K_{21}, K_{12} \cup K_{22})$
or $(K_{11} \cup K_{22}, K_{12} \cup K_{21})$, we have
\[
\mathrm{lk}(\partial A_1, \partial A_2)
= \lambda_L(11,21) + \lambda_L(12,22)
\quad \text{or} \quad
= \lambda_L(11,22) + \lambda_L(12,21).
\]
As we mentioned in Definition~\ref{def:linkage2},
\[
\left| \lambda_L(11,21) + \lambda_L(12,22) \right|
=
\left| \lambda_L(11,22) + \lambda_L(12,21) \right|=2\bar{\lambda}_L(1,2).
\]
This completes the proof. \end{proof}

\begin{lem}
\label{lem:3comp2}
If $L$ is Brunnian and $[\rho(L)] = 0$ in $\mL_0$, then 
$\bar{\lambda}(1,2)$ is even.
\end{lem}

\begin{proof}
If $[\rho(L)] = 0$ in $\mL_0$, then at least one of the following holds:
\begin{itemize}
\setlength{\itemsep}{2mm}
\item
$L$ bounds a disjoint union of  a disk $D$ and an annulus $A$, or
\item
$L$ bounds a disjoint union of a disk $D$ and two M\"{o}bius bands $M_1$, $M_2$.
\end{itemize}
For the first case, we can adopt the same argument as the proof of Lemma~\ref{lem:3comp1},
except for that the boundary of the annulus $A$ is possibly $K_1 \cup -K_2$.
Even for the case, we have 
\[2\bar{\lambda}(1,2)=|\lambda_L(11,21)+\lambda_L(12,22)|=|\lambda_L(11,22)+\lambda_L(12,21)|=0.\]

Next, for the second case, there are two lifts ${M}_{i1}$, ${M}_{i2}$ of $M_i$ in $\Sigma_2(D)$ for each $i \in \{1,2\}$.
Without loss of generality, we may assume that $\partial {M}_{ik} = K_{ik}$. Now,
since $\Sigma_2(D)$ is a rational homology 4-ball, we have
\[
\bar{\lambda}(1,2)=|\lk(K_{11}, K_{21})| = |\lk(\partial{M}_{11}, \partial {M}_{21})| \equiv 
|{M}_{11} \cdot {M}_{21}| = 0 \mod 2.
\]
\end{proof}

The following two lemmas can be proved similarly to Lemmas~\ref{lem:2comp2} and~\ref{lem:2comp3} 
respectively.

\begin{lem}
\label{lem:3comp3}
If $L$ and $L'$ are Brunnian and $[L] = [L'] \neq 0$ in $\tmL_0$, then 
$\bar{\lambda}_L(1,2)=\bar{\lambda}_{L'}(1,2)$. 
\end{lem}

\begin{lem}
\label{lem:3comp4}
If $L$ and $L'$ are Brunnian and $[\rho(L)] = [\rho(L')] \neq 0$ in $\mL_0$, then 
$\bar{\lambda}_L(1,2)=\bar{\lambda}_{L'}(1,2)$. 
\end{lem}

\begin{proof}[Proof of Proposition~\ref{prop:B}]
Figure~\ref{isotopy} describes an isotopy from $B$ to $-B^*$ which sends the triple $(O,K_1,K_2)$ 
to $(-O^*,-K^*_2, -K^*_1)$.
By taking satellite operations, we also have an isotopy from $P_n(B)$ to $-P_n(B)^*$, 
which proves the assertion (1).

To prove the assertions (2) and (3), let us compute $\bar{\lambda}(1, 2)$ for any $P_n(B)$.
For a surface $F$ with $\partial F = O$, and for a basis 
$\alpha = (a_1, a_2)$ illustrated in Figure~\ref{lambda}, we have
\[
G_{\alpha} =
\begin{pmatrix}
0 & -1 \\
-1 &0
\end{pmatrix}.
\]
Hence, by Remark~\ref{cal:linkage}, 
\[\bar{\lambda}_B(1, 2)=
|\lambda_{B}(11,21)|
=\left| (1 \;\; 0)
G_{\alpha}^{-1}
\begin{pmatrix}
0 \\ -1
\end{pmatrix}\right|
= \left|(1 \;\; 0)
\begin{pmatrix}
0 & -1 \\
-1 & 0
\end{pmatrix}
\begin{pmatrix}
0 \\ -1
\end{pmatrix}\right|
= 1.
\]
Furthermore,   
\[\bar{\lambda}_{P_n(B)}(1, 2)=
|\lambda_{P_n(B)}(11,21)|
= \left|(n \;\; 0)
\begin{pmatrix}
0 & -1 \\
-1 & 0
\end{pmatrix}
\begin{pmatrix}
0 \\ -n
\end{pmatrix}\right|
= n^2.
\]
Now, since $w(C_{n,1}) = n$ and $w(C^n) = 2^n$, the assertions (2) and (3) directly follows from Lemmas~\ref{lem:3comp1}, 
\ref{lem:3comp2}, \ref{lem:3comp3} and \ref{lem:3comp4}.
(Note that the primeness follows from Proposition~\ref{prop:Brunnian} and the fact that $C_{n,1}(B)$ and $C^n(B)$ are Brunnian for any $n$.)
\end{proof}

\begin{figure}[htbp]
\begin{center}
\includegraphics[width=.85\linewidth]{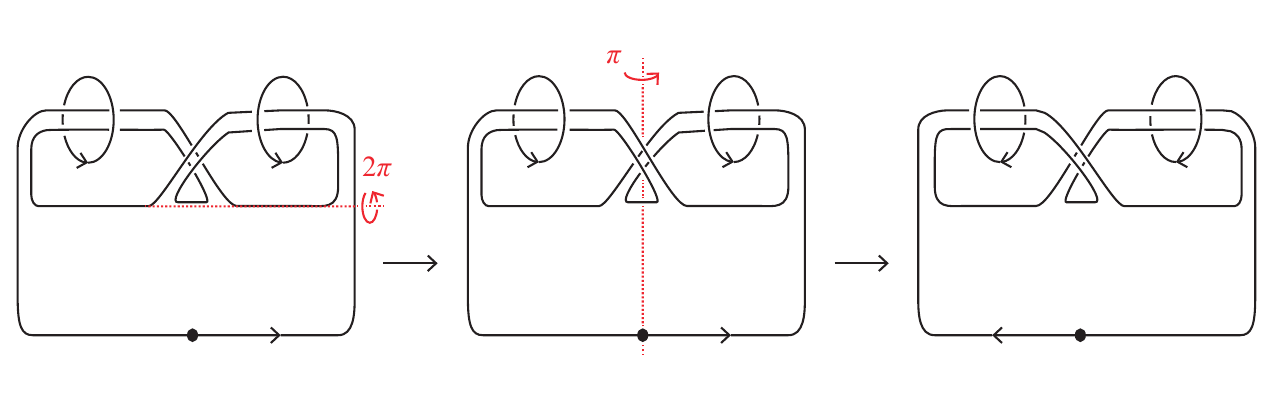}
\caption{An isotopy from $B$ to $-B^*$}\label{isotopy}
\end{center}
\end{figure}

\begin{proof}[Proof of Proposition~\ref{prop:example}]
Since $W_n = C_{n,1}(W)$, $B_n = C_{n,1}(B)$ and $w(C_{n,1}) =n$, 
Proposition~\ref{prop:example} immediately follows from 
Propositions~\ref{prop:W} and \ref{prop:B}.
\end{proof}

\begin{proof}[Proof of Theorem~\ref{thm:main}]
This immediately follows from Theorem~\ref{thm:structure} and Proposition~\ref{prop:example}.
\end{proof}

\begin{proof}[Proof of Theorem~\ref{thm:ker}]
By Theorem~\ref{thm:structure}, we have a direct decomposition
\[
\Ker \rho = (\Ker \rho \cap \tmL_\infty) \oplus (\Ker \rho  \cap \tmL_2)
\]
such that
\begin{itemize}
\setlength{\itemsep}{2mm}
\item
$\Ker \rho \cap \tmL_\infty$ is an free abelian group containing  
$\{C^n(W)~|~n\in\mathbb{N}\}$, and
\item
$\Ker \rho \cap \tmL_2$ is an $\Z/2\Z$-vector space containing  $\{C^n(B)~|~n\in\mathbb{N}\}$.
\end{itemize}
Since $\{C^n(W)~|~n\in\mathbb{N}\}$ (resp.\  $\{C^n(B)~|~n\in\mathbb{N}\}$) is a linearly independent set of 
$\Ker \rho \cap \tmL_\infty$ (resp.\ $\Ker \rho \cap \tmL_2$), 
the equality $\Ker \rho \cong \Z^\infty \oplus (\Z/2\Z)^\infty$ holds.

Next, denote by $\tmL_B$  the subgroup  of $\tmL_0$ generated by $\Ker \rho \cap \mB$,
and let us prove that $\tmL_B$  is a direct summand of $\tmL_0$ isomorphic to $\Z^\infty \oplus (\Z/2\Z)^\infty$.
It follows from Proposition~\ref{prop:Brunnian} that $\mB$ is contained in the set of prime elements in $\tmL$.
Therefore, by Theorem~\ref{thm:structure}, we see that
$\tmL_B \cap \tmL_\infty$ (resp.\ $\tmL_B \cap \tmL_2$) is a direct summand of $\tmL_\infty$ (resp.\ $\tmL_2$)
with a basis $\Ker \rho \cap \mB \cap \tmP_\infty$ (resp.\ $\Ker \rho \cap \mB \cap \tmP_2$).
Moreover, we note that
\begin{itemize}
\setlength{\itemsep}{2mm}
\item
$\Ker \rho  \cap \mB \cap \tmP_\infty$ contains $\{\varepsilon(n) C^n(W)~|~n\in\mathbb{N}\}$ for some $\varepsilon \colon \mathbb{N} \to \{\pm 1\}$, and
\item
$\Ker \rho \cap \mB \cap \tmP_2$ contains $\{C^n(B)~|~n\in\mathbb{N}\}$.
\end{itemize}
As a conclusion, we can say that $\tmL_B$ is a direct summand of $\tmL_0$ and
\[
\tmL_B = (\tmL_B \cap \tmL_\infty) \oplus (\tmL_B \cap \tmL_2) \cong \Z^\infty \oplus (\Z/2\Z)^\infty.
\]
\end{proof}


\begin{thebibliography}{00}

\bibitem{DO} A.\ Donald and B.\ Owens, {\it Concordance groups of links},  Alg.\ Geom.\ Topol.\ {\bf 12}, 2069--2093 (2012). 
https://doi.org/10.2140/agt.2012.12.2069

\bibitem{PY} J.\ H.\ Przytycki and A.\ Yasuhara,
 {\it Linking numbers in rational homology 3-spheres, cyclic branched covers and infinite cyclic covers},
{Trans. Amer Math. Soc.}
{\bf 365}, 3669--3685 (2004). 
https://doi.org/10.1090/S0002-9947-04-03423-3

\bibitem{G} L.\ Goeritz, {\it Knoten und quadratische Formen}, Math.\ Z.\ {\bf 36}, 647-- 654 (1933). 
https://doi.org/10.1007/BF01188642

\bibitem{GL} C.\ McA.\ Gordon and R.\ A.\ Litherland, 
{\it On the signature of a link}, Invent.\ Math.\ {\bf 47}, 53--69  (1978). 
https://doi.org/10.1007/BF01609479
 

\bibitem{Yasuhara} A.\ Yasuhara, {\it Connecting lemmas and representing homology classes of simply connected 4-manifolds}, Tokyo J.\ Math.\ {\bf 19}, no.\ 1, 245--261 (1996). 
https://doi.org/10.3836/tjm/1270043232

\end{thebibliography}
\end{document}